\documentclass{amsart}
\usepackage{amssymb}
\usepackage{tikz}
\usepackage{xcolor}
\definecolor{NotationBlue}{RGB}{0,70,180}

\usepackage{needspace}
\newtheorem{theorem}{Theorem}[section]
\newtheorem{lemma}[theorem]{Lemma}
\newtheorem{proposition}[theorem]{Proposition}
\newtheorem{corollary}[theorem]{Corollary}

 \theoremstyle{definition}
\newtheorem{definition}[theorem]{Definition}

\theoremstyle{remark}
\newtheorem{remark}[theorem]{Remark}

\numberwithin{equation}{section}

\begin{document}
\raggedbottom

\title[]
{Locally uniform ellipticity of the fractional Hessian operators}
\author{Ziyu Gan}
\address{School of Mathematics, Harbin Institute of Technology,
         Harbin, Heilongjiang 150001, China}
\email{24B912023@stu.hit.edu.cn}
\author{Heming Jiao}
\address{School of Mathematics and Institute for Advanced Study in Mathematics, Harbin Institute of Technology,
         Harbin, Heilongjiang 150001, China}
\email{jiao@hit.edu.cn}
\thanks{The second author is supported by the NSFC (Grant No. 12271126) and the Natural Science Foundation of Heilongjiang Province (Grant No. YQ2022A006)}


\begin{abstract}
In \cite{caffarelli2015}, Caffarelli-Charro introduced a fractional Monge-Amp\`{e}re operator. Later, Wu \cite{wu2019} generalized it
to a fractional analogue of $k$-Hessian operators and proved the strict ellipticity for $k=2$. In this paper, we
introduce a fractional analogue of general Hessian operators and prove the stability.
We also show that the fractional analogue $k$-Hessian operators defined in
\cite{wu2019} are uniformly elliptic for all $2\leq k\leq n$.

\end{abstract}

\maketitle
\enlargethispage{3pt}

\section{Introduction}

In this paper, we consider a fractional setting for the Hessian operators
\begin{equation}
\label{GJ-1}
f (\lambda (D^2 u))
\end{equation}
by writing it as a concave envelope of linear operators as in \cite{caffarelli2015} and \cite{wu2019}, where $f$ is a symmetric smooth function of $n$ variables,
$D^2 u$ is the Hessian of a function $u$ defined on $\mathbb{R}^n$ or a subset of $\mathbb{R}^n$ and $\lambda (D^2 u) = (\lambda_1, \ldots, \lambda_n)$ denote the eigenvalues of
$D^2 u$. Following \cite{caffarelli1985}, we assume $f$ to be defined in an open, convex, symmetric cone $\Gamma \subset \mathbb{R}^n$ with vertex at the origin, containing
\[
\Gamma_n := \{\lambda \in \mathbb{R}^n: \mbox{ each component } \lambda_i > 0\} \neq \mathbb{R}^n,
\]
and to satisfy the standard structure conditions throughout the paper:
\begin{equation}
\label{f1}
f_i := \frac{\partial f}{\partial \lambda_i} > 0 \mbox{ in } \Gamma, \ 1 \leq i \leq n,
\end{equation}
\begin{equation}
\label{f2}
f \mbox{ is concave in } \Gamma
\end{equation}
and
\begin{equation}
\label{f3}
f \in C^2 (\Gamma) \cap C^0 (\overline{\Gamma}), f > 0 \mbox{ in } \Gamma \mbox{ and } f = 0
  \mbox{ on } \partial \Gamma.
\end{equation}
Furthermore, we assume that $f$ satisfies
\begin{equation}
\label{f4}
f \mbox{ is homogeneous of degree one.}
\end{equation}
Let $F$ be the function defined by
\[
F (A) = f (\lambda (A))
\]
for a symmetric $n\times n$ matrix $A = \{A_{ij}\}$ with $\lambda (A) \in \Gamma$ and $DF(A)$ denote
the gradient of $F$ at $A$.
For simplicity, we still denote $\Gamma$ to be the set of all symmetric $n\times n$ matrices $A = \{A_{ij}\}$ with $\lambda (A) \in \Gamma$.
Let
\[
F^{ij} (A) = \frac{\partial F}{\partial A_{ij}}(A).
\]
By \eqref{f1} and \eqref{f2}, we find the matrix $\{F^{ij} (A)\}$ is positive definite for any $A \in \Gamma$
and $F$ is also concave in $\Gamma$. (The reader is referred to \cite{Spruck05} for details.)
Denote $\mathcal{S}^{n\times n}$ to be the space of all symmetric $n\times n$ matrices.
We see the directional derivative of $F$ at $A$ in the direction $H \in \mathcal{S}^{n\times n}$ is
\[
D_H F(A)= \frac{d F(A+tH)}{dt}\Big|_{t=0} = \sum_{i,j=1}^n F^{ij}(A)H_{ij}.
\]
By \eqref{f4}, we find
\[
F^{ij} A_{ij} = F (A), \mbox{ for each } A \in \Gamma.
\]
Thus,
for any $A = \{A_{ij}\}, B = \{B_{ij}\} \in \Gamma$, we have
\[
F^{ij} (B) (A_{ij} - B_{ij}) \geq F (A) - F (B)
\]
and
\begin{equation}
\label{GJ-2}
F (A) = \inf_{M \in \mathcal{M}} M^{ij} A_{ij},
\end{equation}
where $\mathcal{M}$ is a subset of all positive definite symmetric $n\times n$ matrices defined by
\[
\mathcal{M} := \{M = \{M^{ij}\}: \mbox{ there exists } T \in \Gamma \mbox{ such that } M^{ij} = F^{ij} (T)\}.
\]
Therefore, for $u \in C^2 (\mathbb{R}^n)$ with $D^2u(x)\in\Gamma$, we have
\begin{equation}
\label{GJ-3}
F (D^2 u (x)) = \inf_{M \in \mathcal{M}} \mathrm{trace} (M D^2 u (x)) = \inf_{M \in \mathcal{M}} \Delta (u \circ \sqrt{M}) (\sqrt{M}^{-1} x).
\end{equation}
Let us recall the definition of the fractional Laplacian (c.f. \cite{chen2020}):
\[
\begin{aligned}
	-(-\Delta)^s u(x) &= c_{n,s} \, \text{P.V.} \int_{\mathbb{R}^n} \frac{u(y) - u(x)}{|y - x|^{n + 2s}} \, dy \\
	&= \frac{c_{n,s}}{2} \, \int_{\mathbb{R}^n} \frac{u(x + y) + u(x - y) - 2u(x)}{|y|^{n + 2s}} \, dy
\end{aligned}
\]
for $0 < s < 1$, where
\[ c_{n,s} = \frac{4^s \Gamma (\frac{n}{2} + s)}
{\pi^{n/2}|\Gamma(-s)|} \]
is a normalization constant. We then define the fractional Hessian operator as \cite{caffarelli2015} and \cite{wu2019} by
\begin{equation}
\label{GJ-4}
\begin{aligned}
   F_{s}[u](x) = \,& \inf_{M \in \mathcal{M}} \Big\{- c_{n,s}^{-1}(-\Delta)^s  (u \circ \sqrt{M}) (\sqrt{M}^{-1} x)\Big\}\\
  = \,& \inf_{M \in \mathcal{M}} \Big\{\text{P.V.} \int_{\mathbb{R}^n}\frac{u(x + y) - u(x)}{|\sqrt{M}^{-1}y|^{n+2s}} \det \sqrt{M}^{-1} \, dy\Big\}\\
   = \,& \inf_{M \in \mathcal{M}} \Big\{ \frac{1}{2} \int_{\mathbb{R}^n} \frac{\delta (u,x,y)}{|\sqrt{M}^{-1}y|^{n+2s}} \det \sqrt{M}^{-1} \, dy \Big\},
\end{aligned}
\end{equation}
where
\[
\delta (u,x,y) := u(x + y) + u(x - y) - 2u(x).
\]
Our first result is the stability of the operator $F_s$.
\begin{theorem}
\label{GJ-thm1}
Assume that $f$ satisfies \eqref{f1}-\eqref{f4}.
Suppose $u$ is admissible (See Definition~\ref{admissible} in Section~2), locally bounded and there exist positive constants $L, R$ such that
\[
|u(x)| \leq L |x|, \ \ \mbox{ for all } |x| \geq R.
\]
Then
\begin{equation}\label{GJ-5}
 \lim_{s\to 1^-}\big((1-s)F_s[u](x)\big)
 =\frac{\omega_n}{4n}F(D^2u(x))
\end{equation}
holds in viscosity sense. (See Definition \ref{limit}.) 
\end{theorem}
In this work, we shall study the existence and regularity of solutions to the Dirichlet problem
\begin{equation}\label{1}
  \begin{cases}
  	F_s[u](x) = u(x) - \phi(x) & \text{in } \mathbb{R}^n \\
  	(u - \phi)(x) \to 0 & \text{as } |x| \to \infty,
  \end{cases}
\end{equation}
where $\phi \in C^{2, \alpha} (\mathbb{R}^n)$ is the prescribed boundary data at infinity. As \cite{caffarelli2015}, $\phi$ is assumed to satisfy that
$\phi$ is strictly convex in compact sets and $\phi = \psi + \eta$ near infinity, with $\psi(x)$ a cone and
\[
|\eta(x)| \leq a |x|^{-\epsilon}, \ |\nabla \eta (x)| \leq a |x|^{-(1+\epsilon)}, \ |D^2 \eta (x)| \leq a |x|^{-(2+\epsilon)}
\]
for some constant $a>0$ and $0<\epsilon <n$. In particular, as $|x| \rightarrow + \infty$,
\[
-(-\Delta)^s \eta (x) = O (|x|^{-(2s+\epsilon)})
\]
and
\[
a_1 |x|^{1-2s} \leq -(-\Delta)^s \psi (x) \leq a_2 |x|^{1-2s},
\]
where $a_1$ and $a_2$ are some positive constants depending on the strict
convexity of the section of the cone $\psi$. $\phi$ is normalized such that $\phi(0)=0$ and $\nabla \phi(0)=0$.

Using the same arguments as in \cite{caffarelli2015} (see \cite{wu2019} also), we have the following existence theorem.
\begin{theorem}
\label{GJ-thm3}
Assume $\phi$ is semi-concave (The reader is referred to
Definition~\ref{semi-concavity} for the definition of semi-concavity)
and Lipschitz continuous, then there exists a unique viscosity solution
to \eqref{1}. Furthermore, the solution is Lipschitz continuous and
semi-concave with the same constants as $\phi$.
\end{theorem}
Typical examples satisfying \eqref{f1}-\eqref{f4} are given by
\[
 f = \sigma_k^{1/k}\quad\text{and}\quad
 f = (\sigma_k/\sigma_l)^{1/(k-l)},\qquad 1\leq l < k \leq n,
\]
defined in the cone
\[
\Gamma_k = \{\lambda \in \mathbb{R}^n: \sigma_j (\lambda) > 0, j =1, \ldots, k\},
\]
where $\sigma_k(\lambda)$ are the elementary symmetric functions
\[
\sigma_k (\lambda) = \sum_{1\leq i_1 < \cdots < i_k \leq n} \lambda_{i_1} \cdots \lambda_{i_k}, \ k= 1, \ldots, n.
\]
In particular, the operator \eqref{GJ-4} coincides with the one defined in \cite{wu2019} as $f = \sigma_k^{1/k}$.

To obtain higher regularity, we usually need to use the nonlocal Evans-Krylov
theory developed in \cite{caffarelli2009,caffarelli2011}.
For $0<\theta\leq\Lambda<\infty$, define
\begin{equation}\label{eq:truncated-operator}
 F_s^{\theta,\Lambda}[u](x)
 =\inf_{\substack{M\in\mathcal M\\ \theta I\leq M\leq\Lambda I}}
 \left\{\frac12\int_{\mathbb R^n}
 \frac{\delta(u,x,y)}{|\sqrt M^{-1}y|^{n+2s}}
 (\det M)^{-1/2}\,dy\right\}.
\end{equation}
The main result is the following theorem.
\begin{theorem}\label{GJ-thm2}
Suppose $f=\sigma_k^{1/k}$, for $2\leq k\leq n$.
Assume that $\frac12<s<1$ and $u$ is Lipschitz continuous,
semi-concave in $\mathbb R^n$, and
\begin{equation}\label{GJ-7}
 (1-s)F_s[u](x)\geq\eta_0>0,\qquad\forall x\in\Omega
\end{equation}
in viscosity sense for some constant $\eta_0>0$. Here $\Omega$ is an open set.
Then there exist constants $0<\theta\leq\Lambda$
depending only on $n,k,\eta_0,s$ and the Lipschitz and semiconcavity
constants of $u$ such that
\begin{equation}\label{GJ-15}
 F_s[u](x)=F_s^{\theta,\Lambda}[u](x),\qquad\forall x\in\Omega.
\end{equation}
For every $x\in\Omega$, the infimum is attained at a matrix
$M_x\in\mathcal M$ satisfying $\theta I\leq M_x\leq\Lambda I$.
Furthermore, Theorem~\ref{GJ-thm2} is stable as $s\to1$, i.e.,
if $u_s$, $s\in[s_0,1)$ with $s_0>1/2$,
satisfy the assumptions with a common $\eta_0>0$ and global Lipschitz
and semiconcavity constants at most $L_*$ and $SC_*$, respectively,
then $\theta,\Lambda$ can be chosen depending only on
$n,k,s_0,\eta_0,L_*,SC_*$, but independently of $s$.
\end{theorem}
By Theorem \ref{GJ-thm2}, Proposition 6.3 in \cite{caffarelli2015} and nonlocal Evans-Krylov theorem proved in \cite{caffarelli2009} and \cite{caffarelli2011},
we obtain the $C^{2s+\alpha}$ estimates for solutions to \eqref{1} with $f = \sigma_k^{1/k}$ and then the solutions are classical.

We remark that when $f= \sigma_n^{1/n}$, \eqref{GJ-1} is the Monge-Amp\`{e}re operator and \eqref{GJ-4} becomes the
fractional Monge-Amp\`{e}re operator defined in \cite{caffarelli2015}. Another class of interesting examples of $f$ are
\[
\tilde{f} (\lambda) = f \Big(\sum_{j\neq 1} \lambda_j, \ldots, \sum_{j\neq n} \lambda_j\Big),
\]
where $f$ satisfies \eqref{f1}-\eqref{f4}. In particular, as $f = \sigma_n^{1/n}$, $\tilde{f}$ is called the $(n-1)$ Monge-Amp\`{e}re
operator. From Remark \ref{GJ-rem2} we see Theorem \ref{GJ-thm2} also holds for the operator
\begin{equation}
\label{GJ-24}
\tilde{f}_k (\lambda) := \sigma_k^{1/k} \Big(\sum_{j\neq 1} \lambda_j, \ldots, \sum_{j\neq n} \lambda_j\Big)
\end{equation}
in
\[
\tilde{\Gamma}_k := \Big\{\lambda \in \mathbb{R}^n: \Big(\sum_{j\neq 1} \lambda_j, \ldots, \sum_{j\neq n} \lambda_j\Big) \in \Gamma_k\Big\}
\]
for $2\leq k \leq n$.

Classical Hessian equations, especially the Monge-Amp\`{e}re equation, appear in many geometric problems as well as the optimal transportation problem.
The Dirichlet problem for Hessian equations was first studied by Ivochkina \cite{Ivochkina81} and Caffarelli-Nirenberg-Spruck \cite{caffarelli1985},
followed by work in \cite{Guan94, Trudinger95, Wang1994}, etc. Guan \cite{Guan23} proved the existence of smooth solutions to the Dirichlet problem for
Hessian equations only under conditions \eqref{f1}-\eqref{f3} and that there exists a subsolution.

Fractional Monge-Amp\`{e}re operator was introduced by
Caffarelli-Charro \cite{caffarelli2015}.\linebreak
Later, Caffarelli-Silvestre \cite{caffarelli2016} defined another nonlocal fractional Monge-Amp\`{e}re operator
\[
  MA_s u(x) = c_{n,s}\inf_{K \in \mathcal{K}_n^s} \left\{ \int_{\mathbb{R}^n} (u(x+y) - u(x) - y \cdot \nabla u(x)) K(y) \, dy  \right\}
\]
by observing that
\[
(\det(D^2 u (x)))^{1/n} = \frac1n
\inf\{\Delta[\tilde{u}\circ\varphi] (0): \mbox{ for all $\varphi$ measure preserving s.t. }\varphi(0)=0\}.
\]
Here
\[
\tilde{u} (y) =  u(x+y) - u(x) - y \cdot \nabla u(x)
\]
and
\[
  \mathcal{K}_n^s = \left\{ K : \mathbb{R}^n \to \mathbb{R}_+ : \left| \{ x \in \mathbb{R}^n : K(x) > r^{-n - 2s} \} \right| = |B_r| \text{ for all } r > 0 \right\}.
\]
Their definition was further generalized by Caffarelli and Soria-Carro \cite{caffarelli2024} recently. A different nonlocal version of the Monge-Amp\`{e}re
operator was also considered in \cite{GuillenSchwab2012}.
Besides, fractional $k$-Hessian operators (the cases $f = \sigma_k^{1/k}$) were introduced and studied by
Wu \cite{wu2019}. In view of \eqref{GJ-3}, the Hessian operators \eqref{GJ-1} can be regarded as a class of degenerate Bellman operators.
It is interesting to consider general fractional Bellman operators.

The rest of this paper is organized as follows. In Section 2, we provide the notations of this paper and some preliminaries.
Section 3 is devoted to proving Theorem \ref{GJ-thm1}. In Section 4, we consider the uniform ellipticity of \eqref{GJ-4} and prove Theorem \ref{GJ-thm2}.

\section{Preliminaries}

In this section we introduce the notations of this paper and recall some basic results and definitions.

For square matrices, \( M > 0 \) means positive definite and \( M \geq 0 \) positive semidefinite.
\( \lambda_{\min}(M) \) and \( \lambda_{\max}(M) \) denote the smallest and largest eigenvalues of $M$, respectively.

We shall denote the \( r \)-dimensional ball of radius 1 and center 0 by \( B_1^r(0) = \{ x \in \mathbb{R}^r : |x| \leq 1 \} \) and the corresponding \( (r - 1) \)-dimensional sphere by \( \partial B_1^r(0) = \{ x \in \mathbb{R}^r : |x| = 1 \} \). Whenever \( r \) is clear from context, we shall simply write \( B_1(0) \) and \( \partial B_1(0) \). \( \mathcal{H}^r \) stands for the \( r \)-dimensional Hausdorff measure.
In the following sections, we denote
$\omega_r$ to be the volume of the unit sphere $| \partial B_1^r (0)|$.

\begin{definition}
\label{semi-concavity}
Let \( A \subset \mathbb{R}^n \) be an open set. We say that a function \( u : A \to \mathbb{R} \) is semiconcave if it is continuous in \( A \) and there exists \( SC \geq 0 \) such that \( \delta(u, x, y) \leq SC|y|^2 \) for all \( x, y \in \mathbb{R}^n \) such that \( [x - y, x + y] \subset A \). The constant \( SC \) is called a semiconcavity constant for \( u \) in \( A \).

Equivalently, \( u(x) - \frac{SC}{2}|x|^2 \) is concave in \( A \). Geometrically, this means that the graph of \( u \) can be touched from above at every point by a paraboloid of the type \( a + \langle b, x \rangle + \frac{SC}{2}|x|^2 \).
\end{definition}
Next, we give the definition of non-smooth admissible functions (in viscosity sense).
\begin{definition}
\label{admissible}
An upper semicontinuous function \( u: \Omega \rightarrow \mathbb{R} \) is called admissible in $\Omega \subset \mathbb{R}^n$ if $\lambda (D^2 \psi (x)) \in \Gamma$ for all $x \in \Omega$ and $\psi \in C^2 (N)$ satisfying $u-\psi$ has a local maximum at $x$, where $N$ is some open neighborhood of $x$ in $\Omega$.

It is easy to find $u \in C^2 (\Omega)$ is admissible in $\Omega$ if and only if $\lambda (D^2 u (x)) \in \Gamma$ for all $x \in \Omega$.
\end{definition}
\begin{remark}
An upper semicontinuous function \( u: \Omega \rightarrow \mathbb{R} \) is called $k$-convex if $u$ satisfies the conditions in Definition \ref{admissible} with the cone $\Gamma$ replaced by $\Gamma_k$.
\end{remark}

We recall from \cite{caffarelli2009} the definition of viscosity solutions.

\begin{definition}
	
	A function \( u : \mathbb{R}^n \to \mathbb{R} \), upper (resp. lower) semicontinuous in \( \overline{\Omega} \), is said to be a subsolution (supersolution) to \( F_{s} [u] = f \), and we write \( F_{s} [u] \geq f \) (resp. \( F_{s} [u] \leq f \)), if every time all the following happen,
	\begin{itemize}
		\item \( x \) is a point in \( \Omega \),
		\item \( N \) is an open neighborhood of \( x \) in \( \Omega \),
		\item \( \psi \) is some \( C^2 \) function in \( \overline{N} \),
		\item \( \psi(x) = u(x) \),
		\item \( \psi(y) > u(y) \) (resp. \( \psi(y) < u(y) \)) for every \( y \in N \setminus \{x\} \),
	\end{itemize}
	and if we let
	\[
	v :=
	\begin{cases}
		\psi & \text{in } N \\
		u & \text{in } \mathbb{R}^n \setminus N,
	\end{cases}
	\]
	then we have \( F_{s} [v](x) \geq f(x) \) (resp. \( F_{s} [v](x) \leq f(x) \)). A solution is a function \( u \) that is both a subsolution and a supersolution.
	
\end{definition}

The following lemma states that \(F_{s} [u] \) can be evaluated classically at those points \( x \) where \( u \) can be touched by a paraboloid. The details of the proof can be found in \cite{caffarelli2015}.

\begin{lemma}\label{lem:classical-evaluation}
	
	Let \( 1/2 < s < 1 \) and \( u : \mathbb{R}^n \to \mathbb{R} \) with asymptotically linear growth. If we have \(F_{s} [u] \geq f \) in \( \mathbb{R}^n \) (resp. \( F_{s} [u] \leq f \)) in the viscosity sense and \( \psi \) is a \( C^2 \) function that touches \( u \) from above (below) at a point \( x \), then \( F_{s} [u](x) \) is defined in the classical sense and \( F_{s} [u](x) \geq f(x) \) (resp. \(F_{s} [u](x) \leq f(x) \)).
	
\end{lemma}

Next, we explain the meaning of \eqref{GJ-5} in viscosity sense.
\begin{definition}
\label{limit}
Suppose $u$ is admissible in $\Omega \subset \mathbb{R}^n$. Let $x \in \Omega$.
We call 
\begin{equation}
\label{viscosity-1}
 \liminf_{s\to1^-}(1-s)F_s[u](x)
 \ge\frac{\omega_n}{4n}F(D^2 u(x))
\end{equation}
in viscosity sense if for any $\psi \in C^2 (B_r (x))$, satisfying that $\psi$ is admissible, $u(x) = \psi (x)$ and $u-\psi$ has a local minimum at $x$,
the function $v$ defined by $v=\psi$ in $B_r(x)$ and
$v=u$ in $\mathbb R^n\setminus B_r(x)$ satisfies
\[
 \liminf_{s\to1^-}(1-s)F_s[v](x)
 \ge\frac{\omega_n}{4n}F(D^2 \psi (x)).
\]

We call
\begin{equation}
\label{viscosity-2}
 \limsup_{s\to1^-}(1-s)F_s[u](x)
 \le\frac{\omega_n}{4n}F(D^2 u(x))
\end{equation}
in viscosity sense if, in addition, $u$ is lower semicontinuous and for any $\psi \in C^2 (B_r (x))$, satisfying that $u(x) = \psi (x)$ 
and $u-\psi$ has a local  maximum at $x$, the function $v$ defined as before satisfies
\[
 \limsup_{s\to1^-}(1-s)F_s[v](x)
 \le\frac{\omega_n}{4n}F(D^2 \psi(x))
\]
If $u$ satisfies both \eqref{viscosity-1} and \eqref{viscosity-2} in viscosity sense, $u$ is called to satisfy \eqref{GJ-5} in viscosity sense.
\end{definition}

To finish this section, we prove an easy but useful proposition.
\begin{proposition}
\label{GJ-prop1}	
Assume that \( 1/2 < s < 1 \) and that $u$ is Lipschitz continuous and semiconcave in $\mathbb{R}^n$. Then for any $x \in \mathbb{R}^n$
and any positive definite $n\times n$ matrix $M$,
we have
\begin{equation}
\label{GJ-16}
\mathcal J_x(M) := \frac{1}{2} \int_{\mathbb{R}^n} \frac{\delta(u, x, y)}{|\sqrt{M}^{-1} y|^{n + 2s}} \det \sqrt{M}^{-1} \, dy < +\infty.
\end{equation}
If, in addition, assume that $\mathcal J_x(M) \geq -K$ for some positive constant $K$, then the integral $\mathcal J_x(M)$ is absolutely convergent, i.e.,
\begin{equation}
\label{GJ-17}
\int_{\mathbb{R}^n} \frac{|\delta(u, x, y)|}{|\sqrt{M}^{-1} y|^{n + 2s}} \det \sqrt{M}^{-1} \, dy < +\infty.
\end{equation}
\end{proposition}
\begin{proof}
Since $u$ is Lipschitz continuous and semiconcave in $\mathbb{R}^n$, there exist positive constants $L$ and $SC$ such that
\[
|u(x)-u(y)| \leq L |x-y|, \ \ \delta (u, x, y) \leq SC |y|^2, \ \mbox{ for all } x, y \in \mathbb{R}^n.
\]
Therefore, for all $\rho>0 $,
\begin{equation}
\label{GJ-18}
\begin{aligned}
	\mathcal J_x(M) = & \frac{1}{2}\int_{\mathbb{R}^n} \frac{\delta(u, x, y)}{|\sqrt{M}^{-1} y|^{n + 2s}} \det \sqrt{M}^{-1} \, dy\\
		= & \frac{1}{2}\int_{\mathbb{R}^n} \frac{\delta(u, x, \sqrt{M} y)}{|y|^{n + 2s}} \, dy \\
		= & \frac{1}{2}\int_{B_{\rho}(0)} \frac{\delta(u, x, \sqrt{M} y)}{|y|^{n + 2s}} \, dy + \frac{1}{2} \int_{\mathbb{R}^n \setminus B_{\rho}(0)} \frac{\delta(u, x, \sqrt{M} y)}{|y|^{n + 2s}} \, dy \\
		\leq & \frac{1}{2} \int_{B_{\rho}(0)} \frac{SC |\sqrt{M} y|^2}{|y|^{n + 2s}} \, dy + \int_{\mathbb{R}^n \setminus B_{\rho}(0)}  \frac{L |\sqrt{M} y|}{|y|^{n + 2s}} \, dy \\
		\leq & \frac{1}{2} \int_{B_{\rho}(0)} \frac{SC \lambda_{\max}(M)}{|y|^{n + 2s - 2}} \, dy + \int_{\mathbb{R}^n \setminus B_{\rho}(0)} \frac{L \lambda_{\max}(M)^{1/2}}{|y|^{n + 2s - 1}} \, dy\\
 = & \frac{\omega_nSC \lambda_{\max}(M)\rho^{2-2s}}{4-4s} + \frac{2\omega_nL \lambda_{\max}(M)^{1/2}\rho^{1-2s}}{4s-2} := K_0
		< +\infty.
\end{aligned}
\end{equation}
Now we further assume that $\mathcal J_x(M) \geq -K$. We note that, similar to \eqref{GJ-18}, by the semiconcavity and Lipschitz continuity of $u$,
\[
\frac{1}{2} \int_{\{\delta(u, x, y) > 0\}} \frac{\delta(u, x, y)}{|\sqrt{M}^{-1} y|^{n + 2s}} \det \sqrt{M}^{-1} \, dy \leq K_0.
\]	
Since
\[
\mathcal J_x(M) = \frac{1}{2} \left(\int_{\{\delta(u, x, y) > 0\}} + \int_{\{\delta(u, x, y) \leq 0\}}\right) \frac{\delta(u, x, y)}{|\sqrt{M}^{-1} y|^{n + 2s}} \det \sqrt{M}^{-1} \, dy\geq -K,
\]
we have
\[
 - \frac{1}{2}\int_{\{\delta(u, x, y) \leq 0\}} \frac{\delta(u, x, y)}{|\sqrt{M}^{-1} y|^{n + 2s}} \det \sqrt{M}^{-1} \, dy \leq K+K_0 < +\infty,
\]
and then $\mathcal J_x(M)$ is absolutely convergent.
\end{proof}

\section{Stability of fractional Hessian operators}

In this section, we consider the stability of the general fractional Hessian operators to prove Theorem \ref{GJ-thm1}.
In this section, $C$ denotes a positive constant under control which may
change from line to line.
\begin{lemma}
\label{GJ-lem1}
Assume that $u$ satisfies the assumptions in Theorem \ref{GJ-thm1}.
Let $1/2<s<1$, $x\in\mathbb{R}^n$, $M\in\mathcal{M}$, and $\rho=\delta_0\lambda_{\max}^{-1/2}(M)$
for some positive constant $\delta_0$. Then there exists a positive constant
$C$ depending only on $n$, $s$, $\|u\|_{L^\infty(B_{R+|x|})}$, $L$, $|x|$ and $\delta_0$
such that
\begin{equation}
\label{GJ-12}
\int_{\mathbb{R}^n\setminus B_\rho(0)}
\frac{\left|\delta\!\left(u,x,\sqrt{M}\,y\right)\right|}
{|y|^{n+2s}}\,dy
\leq C\lambda_{\max}^{s}(M).
\end{equation}
\end{lemma}
\begin{proof}
Let
\[
C_0=\max\left\{
4\|u\|_{L^\infty(B_{R+|x|})},\,2L
\right\}.
\]
We see
\[
|\delta(u,x,z)|
\leq C_0(1+|z|+|x|)
\leq C_x(1+|z|),
\]
where
\[
C_x:=C_0(1+|x|),
\]
for all $z\in\mathbb{R}^n$. Therefore,
\[
\begin{aligned}
\int_{\mathbb{R}^n\setminus B_\rho(0)}
\frac{\left|\delta\!\left(u,x,\sqrt{M}\,y\right)\right|}
{|y|^{n+2s}}\,dy
&\leq
C_x\int_{\mathbb{R}^n\setminus B_\rho(0)}
\frac{1+\lambda_{\max}^{1/2}(M)|y|}
{|y|^{n+2s}}\,dy \\
&=
C_x \omega_n
\left(
\frac{\rho^{-2s}}{2s}
+
\lambda_{\max}^{1/2}(M)
\frac{1}{2s-1}\rho^{1-2s}
\right) \\
&=
C_x \omega_n
\left(
\frac{\delta_0^{-2s}}{2s}
+
\frac{\delta_0^{1-2s}}{2s-1}
\right)
\lambda_{\max}^{s}(M).
\end{aligned}
\]
Choosing
\[
C=
C_x \omega_n
\left(
\frac{\delta_0^{-2s}}{2s}
+
\frac{\delta_0^{1-2s}}{2s-1}
\right),
\]
we get \eqref{GJ-12}.
\end{proof}

\begin{proof}[Proof of Theorem \ref{GJ-thm1}.]
We first consider the case $u \in C^2$.	
By \eqref{GJ-3}, we have
\[
F\left(D^{2} u(x)\right) = \inf _{M \in \mathcal{M}} \left\{ \text{trace}\left(M D^{2} u(x)\right) \right\}
\]
provided $u\in C^2$ and \(\lambda(D^{2} u(x)) \in \Gamma\).
Thus, it suffices to prove:
\[
\lim _{s \to 1^-} (1-s) F_{s}[u](x) = \frac{\omega_n}{4n} \inf _{M \in \mathcal{M}} \left\{ \text{trace}\left(M D^{2} u(x)\right) \right\}.
\]

For suitable \(0 < \rho < \infty\) and fixed $x$, we have
\begin{equation}\label{4}
\begin{aligned}
	& \int_{\mathbb{R}^{n}} \frac{\delta(u, x, y)}{\left|\sqrt{M}^{-1} y\right|^{n+2 s}} \det \sqrt{M}^{-1} \, dy\\
	= & \int_{\mathbb{R}^{n}} \frac{\delta(u, x, \sqrt{M} y)}{|y|^{n+2 s}} \, dy \\
	= & \int_{B_{\rho}(0)} \frac{\left\langle D^{2} u(x) \cdot \sqrt{M} y, \sqrt{M} y\right\rangle}{|y|^{n+2 s}} \, dy + \int_{\mathbb{R}^{n} \setminus B_{\rho}(0)} \frac{\delta(u, x, \sqrt{M} y)}{|y|^{n+2 s}} \, dy\\
	&+ \int_{B_{\rho}(0)} \frac{\delta(u, x, \sqrt{M} y) - \left\langle D^{2} u(x) \cdot \sqrt{M} y, \sqrt{M} y\right\rangle}{|y|^{n+2 s}} \, dy.
\end{aligned}
\end{equation}
For the first term on the right hand side, we have
\begin{equation}\label{5}
\begin{aligned}
	& \int_{B_{\rho}(0)} \frac{\left\langle D^{2} u(x) \cdot \sqrt{M} y, \sqrt{M} y\right\rangle}{|y|^{n+2 s}} \, dy\\
	= & \int_{0}^{\rho} \int_{\partial B_{1}(0)} \frac{ \left\langle D^{2} u(x) \cdot \sqrt{M} y, \sqrt{M} y\right\rangle}{r^{2s-1}} \, dS_{y} dr \\
	= & \frac{\rho^{2-2 s}}{2-2 s} \int_{\partial B_{1}(0)} \left\langle D^{2} u(x) \cdot \sqrt{M} y, \sqrt{M} y\right\rangle \, dS_{y} \\
	= & \frac{\rho^{2-2 s}}{2-2 s} \int_{B_{1}(0)} \text{div}_y \left[\sqrt{M}^T D^{2} u(x) \sqrt{M}  y\right] \, dy \\
	= & \frac{\rho^{2-2 s}}{n(2-2 s)} \omega_{n} \text{trace}\left(M D^{2} u(x)\right),
\end{aligned}
\end{equation}
where we have used the divergence theorem for the penultimate equality.
Next, since
\[
\delta(u, x, \sqrt{M} y) = \left\langle D^{2} u(x) \cdot \sqrt{M} y, \sqrt{M} y\right\rangle + o\left(y^TMy\right) \quad \text{as} \quad |y| \to 0,
\]
we have, for any fixed $\varepsilon$ small,
\[
\left| \delta(u, x, \sqrt{M} y) - \left\langle D^{2} u(x) \cdot \sqrt{M} y, \sqrt{M} y\right\rangle \right|
\leq \varepsilon \lambda_{\text{max}}(M) |y|^{2}
\]
provided $|y| < \rho$ and $0 < \rho = \delta_0\frac{1}{\lambda_{\text{max}}^{1/2}(M)}$ for some $\delta_0 > 0$ sufficiently small.
It follows that
\begin{equation}\label{6}
\begin{aligned}
	& \int_{B_{\rho}(0)} \frac{\left|\delta(u, x, \sqrt{M} y) - \left\langle D^{2} u(x) \cdot \sqrt{M} y, \sqrt{M} y\right\rangle \right|}{|y|^{n+2s}} \, dy\\
	\leq & \int_{B_{\rho}(0)} \frac{\varepsilon \lambda_{\text{max}}(M)}{|y|^{n+2s-2}} \, dy \\
	= & \frac{\varepsilon \omega_n}{2-2s} \rho^{2-2 s} \lambda_{\text{max}}(M).
\end{aligned}
\end{equation}

Combining \eqref{4}, \eqref{5}, \eqref{6} and \eqref{GJ-12}, we obtain
\[
\begin{aligned}
	& \int_{\mathbb{R}^{n}} \frac{\delta(u, x, y)}{|\sqrt{M}^{-1} y|^{n+2 s}} \det \sqrt{M}^{-1} \, dy\\
	\leq & \frac{\rho^{2-2 s}}{n(2-2 s)} \omega_n \text{trace}\left(M D^{2} u(x)\right) + \frac{\varepsilon \omega_n}{2-2 s} \rho^{2-2 s} \lambda_{\text{max}}(M)
           + C \lambda_{\text{max}}^s (M).
\end{aligned}
\]
By the definition of \(F_{s}[u](x)\),
\[
\begin{aligned}
	F_{s}[u](x)
	&\leq \frac{\rho^{2-2 s} \omega_n}{4n(1-s)} \text{trace}\left(M D^{2} u(x)\right) + \frac{\varepsilon \omega_n}{4(1-s)} \rho^{2-2 s} \lambda_{\text{max}}(M)+ C \lambda_{\text{max}}^s (M).
\end{aligned}
\]
Multiplying both sides by \((1-s)\),
\begin{equation}
\label{GJ-8}
\begin{aligned}
	& (1-s) F_{s}[u](x)\\
	\leq & \frac{\omega_n}{4n} \rho^{2-2 s} \text{trace}\left(M D^{2} u(x)\right) + \frac{\varepsilon \omega_n}{4} \rho^{2-2 s} \lambda_{\text{max}}(M) + C (1-s) \lambda_{\text{max}}^s (M).
\end{aligned}
\end{equation}
Taking the limit as \(s \to 1^-\), we get
\[
\limsup_{s \to 1^-} \left( (1-s) F_{s}[u](x) \right)
\leq \frac{\omega_n}{4n} \text{trace}\left(M D^{2} u(x)\right) + \frac{\varepsilon\omega_n}{4} \lambda_{\text{max}}(M).
\]
By the arbitrariness of \(\varepsilon\),
\[
\limsup_{s \to 1^-} \left( (1-s) F_{s}[u](x) \right)
\leq \frac{\omega_n}{4n} \text{trace}\left(M D^{2} u(x)\right).
\]
Since this holds for all \(M \in \mathcal{M}\), we have
\[
\limsup_{s \to 1^-} \left( (1-s) F_{s}[u](x) \right)
\leq \frac{\omega_n}{4n} \inf _{M \in \mathcal{M}} \text{trace}\left(M D^{2} u(x)\right).
\]

Next, we prove the reverse inequality:
\begin{equation}
\label{GJ-14}
\liminf_{s \to 1^-} \left( (1-s) F_{s}[u](x) \right)
\geq \frac{\omega_n}{4n} \inf _{M \in \mathcal{M}} \left\{ \text{trace}\left(M D^{2} u(x)\right) \right\}.
\end{equation}
From the earlier computation \eqref{4},
\[
\begin{aligned}
	& \int_{B_{\rho}(0)} \frac{\left\langle D^{2} u(x) \cdot \sqrt{M} y, \sqrt{M} y\right\rangle}{|y|^{n+2 s}} \, dy\\
	= & \int_{\mathbb{R}^{n}} \frac{\delta(u, x, \sqrt{M} y)}{|y|^{n+2 s}} \, dy
	- \int_{B_{\rho}(0)} \frac{\delta(u, x, \sqrt{M} y) - \left\langle D^{2} u(x) \cdot \sqrt{M} y, \sqrt{M} y\right\rangle}{|y|^{n+2 s}} \, dy \\
	& - \int_{\mathbb{R}^{n} \setminus B_{\rho}(0)} \frac{\delta(u, x, \sqrt{M} y)}{|y|^{n+2 s}} \, dy.
\end{aligned}
\]
By \eqref{5}, \eqref{6} and \eqref{GJ-12} again, for any \(M \in \mathcal{M}\), we have
\begin{equation}
\label{GJ-13}
\begin{aligned}
	\frac{\omega_n}{n(2-2 s)} \rho^{2-2 s} \text{trace}\left(M D^{2} u(x)\right)
	\leq \,& \int_{\mathbb{R}^{n}} \frac{\delta(u, x, \sqrt{M} y)}{|y|^{n+2 s}} \, dy\\
	& + \varepsilon \frac{\omega_n}{2(1-s)} \rho^{2-2 s} \lambda_{\text{max}}(M) + C \lambda_{\text{max}}^s (M).
\end{aligned}
\end{equation}
By the definition of \(F_{s}[u](x)\), for any \(\varepsilon' > 0\), there exists \(M_{s,\varepsilon'} \in \mathcal{M}\) such that
\[
\frac{1}{2} \int_{\mathbb{R}^{n}} \frac{\delta\left(u, x, \sqrt{M_{s,\varepsilon'}} y\right)}{|y|^{n+2 s}} \, dy
\leq F_{s}[u](x) + \varepsilon'.
\]
Substituting it into the inequality \eqref{GJ-13}, we get
\begin{equation}
\label{GJ-6}
\begin{aligned}
	\frac{\omega_n}{4n(1-s)} \rho^{2-2 s} \text{trace}\left(M_{s,\varepsilon'} D^{2} u(x)\right)
	\leq \,& F_{s}[u](x) + \varepsilon' + \frac{C}{2} \lambda_{\text{max}}^s (M_{s,\varepsilon'})\\
	&+ \varepsilon \frac{\omega_n}{4(1-s)} \rho^{2-2 s} \lambda_{\text{max}}(M_{s,\varepsilon'}).
\end{aligned}
\end{equation}
To proceed we prove that there exist positive constants $s_0 \in (1/2, 1)$ close to 1 and
$C_1$ depending only on $n$, $s_0$, $L$, $|D^2 u (x)|$ and
a $L^\infty$ bound of $u$ on a compact subset such that for any $s \geq s_0$,
we have
\begin{equation}
\label{GJ-9}
\lambda_{\text{max}}(M_{s,\varepsilon'}) \leq C_1
\end{equation}
provided $\varepsilon'$ is bounded. Indeed, by choosing a fixed $M_0 \in \mathcal{M}$ (Say, e.g. $M_0 = D F(I)$.) in \eqref{GJ-8},
we obtain, for any $s \geq s_0$,
\begin{equation}
\label{GJ-10}
(1-s) F_s [u](x) \leq C
\end{equation}
for some constant $C$ depending on $n$, $s_0$, $L$, $|D^2 u (x)|$ and
a $L^\infty$ bound of $u$ in a compact subset. Next, since $D^2 u (x) \in \Gamma$, there exists a positive constant $\delta_1$ such that
$D^2 u (x) - \delta_1 \xi \otimes \xi \in \Gamma$, where $\xi$ is a unit eigenvector of the eigenvalue $\lambda_{\text{max}} (M_{s,\varepsilon'})$.
By \eqref{f2}, \eqref{f3} and \eqref{f4} we see
\begin{equation}
\label{GJ-11}
\text{trace}(M_{s,\varepsilon'} D^2 u (x)) \geq F (D^2 u (x) - \delta_1 \xi \otimes \xi) + \delta_1 M_{s,\varepsilon'}^{ij}\xi_i\xi_j
 \geq \delta_1 \lambda_{\text{max}} (M_{s,\varepsilon'}).
\end{equation}
Combing \eqref{GJ-6}, \eqref{GJ-10} and \eqref{GJ-11}, we have
\[
\begin{aligned}
(\delta_1-n\varepsilon) \lambda_{\text{max}} (M_{s,\varepsilon'}) \leq \,& C \rho^{2s-2} [1+(1-s)\lambda_{\max}^s(M_{s,\varepsilon'})]\\
 \leq \,& C [\lambda_{\max}^{1-s}(M_{s,\varepsilon'}) + (1-s)\lambda_{\max}(M_{s,\varepsilon'})]
\end{aligned}
\]
Choosing $s_0$ such that $1-s_0 \ll \delta_1$, we get \eqref{GJ-9} for any $1 > s \geq s_0$ provided $\varepsilon \ll \delta_1/n$.
Thus, by \eqref{GJ-6} and \eqref{GJ-9}, we obtain
\[
\begin{aligned}
	& \frac{\omega_n}{4n(1-s)} \rho^{2-2 s} \inf _{M \in \mathcal{M}} \left\{ \text{trace}\left( M D^{2} u(x) \right) \right\}\\
	\leq \,& F_{s}[u](x) + \varepsilon'
	+ \varepsilon \frac{\omega_n}{4(1-s)} \rho^{2-2 s} C_1 +\frac{CC^s_1}{2}.
\end{aligned}
\]
Multiplying both sides by \((1-s)\) and taking the limit as \(s \to 1^-\),
\[
\frac{\omega_n}{4n} \inf _{M \in \mathcal{M}} \left\{ \text{trace}\left(M D^{2} u(x)\right) \right\}
\leq \liminf_{s \to 1^-} ((1-s) F_{s}[u](x)) + \varepsilon \frac{\omega_n}{4} C_1.
\]
By the arbitrariness of \(\varepsilon\), we get \eqref{GJ-14}.
Thus, the equality \eqref{GJ-5} holds for any admissible function \(u \in C^{2}\).

Next, we consider the case that $u$ is non-smooth.
Let \(\psi\) be any function satisfying that
\[
\psi(x) = u(x); \quad \psi(y) > u(y) \ (\text{or} \ \psi(y) < u(y)) \ \forall y \in \mathcal{N} \setminus \{x\},
\]
where \(\mathcal{N}\) is a neighborhood of \(x\) and \(\psi \in C^{2}(\overline{\mathcal{N}})\). Define
\[
v(z) :=
\begin{cases}
	\psi(z), & z \in \mathcal{N} \\
	u(z), & z \in \mathbb{R}^{n} \setminus \mathcal{N}
\end{cases}
\]
Taking \(x+\sqrt M B_{\rho}(0) \subset \mathcal{N}\), similar to the classical proof, we have
\[
\begin{aligned}
	& \int_{\mathbb{R}^{n}} \frac{\delta(v, x, y)}{\left|\sqrt{M}^{-1} y\right|^{n+2s}} \det \sqrt{M}^{-1} \, dy\\
	= & \int_{B_{\rho}(0)} \frac{\left\langle D^{2} \psi(x) \sqrt{M} y, \sqrt{M} y\right\rangle}{|y|^{n+2 s}} \, dy \\
	&+ \int_{B_{\rho}(0)} \frac{\delta(\psi, x, \sqrt{M} y) - \left\langle D^{2} \psi(x) \sqrt{M} y, \sqrt{M} y\right\rangle}{|y|^{n+2 s}} \, dy \\
	& + \int_{\mathbb{R}^{n} \setminus B_{\rho}(0)} \frac{\delta(v, x, \sqrt{M} y)}{|y|^{n+2 s}} \, dy.
\end{aligned}
\]
By analogous arguments to the classical case, we can prove
\[
\liminf_{s \to 1^-} \left( (1-s) F_{s}[v](x) \right)
\geq \frac{\omega_n}{4n} \inf _{M \in \mathcal{M}} \text{trace}\left(M D^{2} \psi(x)\right)
\]
and
\[
\limsup_{s \to 1^-} \left( (1-s) F_{s}[v](x) \right)
\leq \frac{\omega_n}{4n} \inf _{M \in \mathcal{M}} \text{trace}\left(M D^{2} \psi(x)\right).
\]
Thus, \eqref{GJ-5} holds in viscosity sense.
\end{proof}

\section{Uniform ellipticity}\label{sec:strict-ellipticity}

In this section, we consider the uniform ellipticity of the case $f=\sigma_k^{1/k}$ for $2\leq k \leq n$. In particular,
we prove Theorem~\ref{GJ-thm2}. As in the preceding sections, $C$ denotes a positive constant which may
change from line to line.

\begin{lemma}\label{lem:normalized-matrix-family}
Suppose $f$ satisfies \eqref{f4}. Let $F(T)=f(\lambda(T))$, $T\in\Gamma$,
and
\[
 \Sigma=\{T\in\Gamma:F(T)=1\}.
\]
Then
\begin{equation}\label{eq:normalized-matrix-family}
 \mathcal M=\{M=\{M^{ij}\}:\text{ there exists }T\in\Sigma
 \text{ such that }M^{ij}=F^{ij}(T)\}.
\end{equation}
\end{lemma}

\begin{proof}
Since $f$ is homogeneous of degree one, i.e.,
\[
 F(aT)=aF(T),\qquad T\in\Gamma,\quad a>0,
\]
we have
\[
 \{F^{ij}(aT)\}=\{F^{ij}(T)\}
\]
for any $T\in\Gamma$ and $a>0$.
Since $F(T)>0$, taking $a=1/F(T)$ gives $T/F(T)\in\Sigma$.
Thus, \eqref{eq:normalized-matrix-family} follows immediately.
\end{proof}

Let $\mathcal{S}_+^{n\times n}=\{M=M^T:M>0\}$ be the cone of $n\times n$
positive definite symmetric matrices. Fix $x\in\mathbb R^n$.
Define
\[
 \mathcal J_x(M)=\frac12\int_{\mathbb R^n}\delta(u,x,y)K_M(y)\,dy,\
\mbox{ for } K_M(y)=\frac{(\det M)^{-1/2}}{|M^{-1/2}y|^{n+2s}}.
\]
Next, by Lemma \ref{lem:classical-evaluation} and Proposition \ref{GJ-prop1}, we have
\begin{lemma}\label{lem:viscosity-classical-lower}
Suppose that $1/2<s<1$ and $u$ is semiconcave and Lipschitz continuous
in $\mathbb R^n$. Assume that
\begin{equation}\label{eq:classical-lower-bound}
 (1-s)F_s[u]\geq\eta_0>0
\end{equation}
in viscosity sense in $\Omega$.
Then, for every $x\in\Omega$, \eqref{eq:classical-lower-bound} holds in classical sense
and the integral $\mathcal J_x(M)$
is absolutely convergent for each $M\in\mathcal M$.
\end{lemma}
The next lemma shows that under conditions of Lemma \ref{lem:viscosity-classical-lower}, $\mathcal J_x(M)$ is absolutely convergent
not only for $M\in\mathcal M$ but also for all $M\in \mathcal S_+^{n\times n}$.
\begin{lemma}\label{lem:bellman-variation}
Suppose that $1/2<s<1$ and $u$ is semiconcave and Lipschitz continuous
in $\mathbb R^n$, with constants $SC$ and $L$, respectively. Assume that \eqref{eq:classical-lower-bound}
holds in viscosity sense. Then for each fixed $x\in \Omega$,
the integral $\mathcal J_x(M)$ is absolutely convergent
for each $M\in\mathcal S_+^{n\times n}$. Furthermore, $\mathcal J_x\in C^1(\mathcal S_+^{n\times n})$
and, for any symmetric matrix $H$, the directional derivative of $\mathcal J_x$ at $M$ in the direction $H$ is given by
\[
 D_H\mathcal J_x(M)
 =\frac12\int_{\mathbb R^n}\delta(u,x,y)D_HK_M(y)\,dy,
\]
where
\[
 D_HK_M(y)=K_M(y)\left[
 -\frac12\operatorname{trace}(M^{-1}H)
 +\frac{n+2s}{2}\frac{y^TM^{-1}HM^{-1}y}{y^TM^{-1}y}
 \right].
\]
\end{lemma}

\begin{proof}
For any symmetric matrix $H$, we have
\[
 \begin{aligned}
 D_H\log\det M
 &=\left.\frac{d}{dr}\log\det(M+rH)\right|_{r=0}
 =\operatorname{trace}(M^{-1}H),\\
 D_HM^{-1}
 &=\left.\frac{d}{dr}(M+rH)^{-1}\right|_{r=0}
 =-M^{-1}HM^{-1}.
 \end{aligned}
\]
Therefore,
\[
 D_H(y^TM^{-1}y)
 =\left.\frac{d}{dr}\big(y^T(M+rH)^{-1}y\big)\right|_{r=0}
 =-y^TM^{-1}HM^{-1}y.
\]
Since
\[
 \log K_M(y)=-\frac12\log\det M
 -\frac{n+2s}{2}\log(y^TM^{-1}y),
\]
we have
\[
 \begin{aligned}
 D_H\log K_M(y)
 &=\left.\frac{d}{dr}\log K_{M+rH}(y)\right|_{r=0}\\
 &=-\frac12\operatorname{trace}(M^{-1}H)
 +\frac{n+2s}{2}\frac{y^TM^{-1}HM^{-1}y}{y^TM^{-1}y}.
 \end{aligned}
\]
Thus,
\[
 \begin{aligned}
 D_HK_M(y)
 &=D_H\exp(\log K_M(y))\\
 &=K_M(y)\left[
 -\frac12\operatorname{trace}(M^{-1}H)
 +\frac{n+2s}{2}\frac{y^TM^{-1}HM^{-1}y}{y^TM^{-1}y}
 \right].
 \end{aligned}
\]
For any fixed $M_0\in\mathcal S_+^{n\times n}$, there exists a positive constant
$C_0>1$ such that $C_0^{-1}I<M_0<C_0I$.
We see that, for any $M\in \mathcal{S}^{n\times n}$ in the neighborhood $U_{M_0} = \{M \in \mathcal{S}^{n\times n}: C_0^{-1}I\leq M\leq C_0I\}$ and any symmetric matrix
$H$ with $|H| = 1$,
\begin{equation}\label{eq:kernel-directional-domination}
 |K_M(y)|+|D_HK_M(y)|
 \leq\frac{C}{|y|^{n+2s}},\qquad C=C(n,s,C_0).
\end{equation}
Indeed, $C_0^{-n}\leq\det M\leq C_0^n$,
$C_0^{-1}|y|^2\leq y^TM^{-1}y\leq C_0|y|^2$, and
\[
 |\operatorname{trace}(M^{-1}H)|\leq\sqrt n\,C_0|H|,
 \qquad
 \frac{|y^TM^{-1}HM^{-1}y|}{y^TM^{-1}y}\leq C_0|H|.
\]
Next, by the symmetry and homogeneity of $f$, we have
\[
 F^{ij}(I)=f_i(1,\ldots,1)\delta_{ij}
 =\frac{F(I)}n\delta_{ij}=:a\delta_{ij},\qquad a>0,
\]
which means $aI\in\mathcal M$.
By \eqref{eq:classical-lower-bound}, we have
\[
 \mathcal J_x(aI)\geq\inf_{M\in\mathcal M}\mathcal J_x(M)
 \geq\frac{\eta_0}{1-s}.
\]
Note that
\[
 \mathcal J_x(aI)
 =\frac{a^s}{2}\int_{\mathbb R^n}
 \frac{\delta(u,x,y)}{|y|^{n+2s}}\,dy.
\]
By Proposition \ref{GJ-prop1}, we have
\begin{equation}\label{eq:absolute-delta}
 \int_{\mathbb R^n}\frac{|\delta(u,x,y)|}{|y|^{n+2s}}\,dy
 < + \infty.
\end{equation}
By \eqref{eq:kernel-directional-domination} and \eqref{eq:absolute-delta}, using the dominated convergence theorem,
we see $\mathcal J_x (M)$ and
\[
\int_{\mathbb R^n} \delta(u,x,y) D_H K_M (y) \,dy
\]
are both absolutely convergent for any $M\in U_{M_0}$. Thus, $\mathcal J_x$ is differentiable at $M_0$ and
\[
D_H \mathcal J_x (M_0) = \frac{1}{2}\int_{\mathbb R^n} \delta(u,x,y) D_H K_{M_0} (y) \,dy.
\]
By the dominated convergence theorem again, we see that $D_H \mathcal J_x$ is continuous at $M_0$.
Since $M_0$ is arbitrary, $\mathcal J_x\in C^1(\mathcal S_+^{n\times n})$.
\end{proof}

\begin{lemma}\label{lem:bellman-directional-bound}
Let $u$ satisfy the assumptions of Lemma \ref{lem:bellman-variation}. Fix $x\in\Omega$.
Suppose $M \in \mathcal{S}^{n\times n}_+$ and $m = \lambda_{\min} (M)$. Assume that $e \in \mathbb{R}^n$ is a unit eigenvector of $m$, i.e.,
$Me=me$ and $|e|=1$. Then
\begin{equation}\label{eq:directional-bound}
 e^T D\mathcal J_x(M)e\leq \frac{C}{1-s} m^{s-1},
\end{equation}
for some positive constant $C$ depending only on $n$, $s$, $L$, and $SC$, and $C$ remains bounded as $s\to1$.
Moreover, for every positive definite matrix $M$ and $t>0$,
\begin{equation}\label{eq:J-homogeneity}
 \mathcal J_x(tM)=t^s\mathcal J_x(M),\qquad
 \operatorname{trace}\big(D\mathcal J_x(M)M\big)=s\mathcal J_x(M).
\end{equation}
\end{lemma}

\begin{proof}

First we note that
\begin{equation}
\label{GJ-20}
\begin{aligned}
 e^TDK_M(y)e
 &=K_M(y)\left[-\frac12e^TM^{-1}e
 +\frac{n+2s}{2}\frac{(e^TM^{-1}y)^2}{y^TM^{-1}y}\right]\\
 &=K_M(y)\left[-\frac12m^{-1}
 +\frac{n+2s}{2}\frac{m^{-2}(e^Ty)^2}{y^TM^{-1}y}\right].
\end{aligned}
\end{equation}
Let $m=m_1 \leq m_2 \leq \cdots \leq m_n$ be all the eigenvalues of $M$ and
$\mu_j=1/\sqrt{m_j}$, $j=1,\ldots,n$. We see $\mu_1 \geq \cdots \geq \mu_n$ are eigenvalues of $\sqrt{M}^{-1}$.
Let
\[
J :=  \frac12\int_{\mathbb R^n}\delta(u,x,y)
 \frac{(e^T y)^2}{y^TM^{-1}y}K_M(y)\,dy.
\]
By \eqref{GJ-20}, we get
\begin{equation}\label{eq:directional-cancellation}
 \begin{aligned}
 e^TD\mathcal J_x(M)e
 &=-\frac{\mu_1^2}{2}\mathcal J_x(M)
 +\frac{n+2s}{2}\mu_1^4 J\\
 &=\frac{n+2s}{2}\mu_1^4
 \left(J-\frac{\mathcal J_x(M)}{(n+2s)\mu_1^2}\right).
 \end{aligned}
\end{equation}
Since \(M\) is symmetric, there exists an orthogonal matrix \(P\) with its first column being $e$, such that
\[
P^T M P = \mbox{diag} \{m_1, \ldots, m_n\}.
\]
Let \(\tilde{\delta}(u,x,y) = u(x+Py)+u(x-Py)-2u(x)\), we have
\[
J =\frac12\int_{\mathbb R^n}\tilde{\delta}(u,x,y)
 \frac{y_1^2\prod_{j=1}^n\mu_j}
 {(\mu_1^2y_1^2+\cdots+\mu_n^2y_n^2)^{1+\frac{n+2s}{2}}}\,dy.
\]
Note that
\[
 \partial_{y_1}\big(y_1K_M(y)\big)
 =K_M(y)-\frac{(n+2s)\mu_1^2 y_1^2}
 {\mu_1^2y_1^2+\cdots+\mu_n^2y_n^2}K_M(y).
\]
We have
 \begin{equation}
 \label{GJ-19}
 J-\frac{\mathcal J_x(M)}{(n+2s)\mu_1^2}
 =-\frac1{2(n+2s)\mu_1^2}
 \int_{\mathbb R^n}
 \tilde{\delta}(u,x,y)\,\partial_{y_1}\big(y_1K_M(y)\big)\,dy.
 \end{equation}
By the semiconcavity and Lipschitz continuity of $u$, we find
\[
 \partial_e u(a+te)-\partial_e u(a-te)
 \leq2\min\{SC\,t,L\}
\]
for each $a\in\mathbb R^n$ and almost every $t>0$, where the derivatives
are taken along lines parallel to $e$. Therefore, for each fixed $(y_2,\ldots,y_n)\ne0$, integrating
by parts and pairing $y_1$ with $-y_1$, we obtain
\[
 \begin{aligned}
 -\int_{\mathbb R}
 \tilde{\delta}(u,x,y)\,\partial_{y_1}\big(y_1K_M(y)\big)\,dy_1
 = & \int_{\mathbb R}
 \partial_{y_1}\tilde{\delta}(u,x,y)\,y_1K_M(y)\,dy_1\\
 \leq & 4\int_0^\infty
 \min\{SC\,y_1^2,Ly_1\}\,K_M(y)\,dy_1,
 \end{aligned}
\]
where the boundary terms vanish by the Lipschitz continuity of $u$ and the fact $K_M (y) \leq C|y_1|^{-(n+2s)}$. Since
$|\partial_{y_1}(y_1K_M)|\leq C K_M$, applying Fubini's theorem to \eqref{GJ-19}, we get
\[
 J-\frac{\mathcal J_x(M)}{(n+2s)\mu_1^2}
 \leq\frac1{(n+2s)\mu_1^2}
 \int_{\mathbb R^n}\min\{SC\,y_1^2,L|y_1|\}K_M(y)\,dy.
\]
Let
\[
 \begin{cases}
 z_1=y_1,\\[2pt]
 z_j=\dfrac{\mu_jy_j}{\mu_1|y_1|},
 \qquad j=2,\ldots,n.
 \end{cases}
\]
We see
\[
 dz=\frac{\mu_2\cdots\mu_n}{(\mu_1|y_1|)^{n-1}}\,dy.
\]
Writing $z'=(z_2,\ldots,z_n)$, we have
\begin{equation}\label{eq:handwritten-integral-split}
 \begin{aligned}
J-\frac{\mathcal J_x(M)}{(n+2s)\mu_1^2} \leq \,& \frac1{(n+2s)\mu_1^2}
 \int_{\mathbb R^n}\frac{\min\{SC\,z_1^2,L|z_1|\}\mu_1^n |z_1|^{n-1}}{[\mu_1^2 z_1^2 (1+|z'|^2)]^{\frac{n+2s}{2}}}\,dz\\
  = \frac{1}{(n+2s)\mu_1^{2+2s}} & \int_{\mathbb{R}} \frac{\min\{SC\,z_1^2,L|z_1|\}}{|z_1|^{2s+1}}dz_1 \int_{\mathbb{R}^{n-1}} \frac{dz'}{(1+|z'|^2)^{\frac{n+2s}{2}}}.
 \end{aligned}
\end{equation}
To proceed we see
\begin{equation}
\label{GJ-21}
\begin{aligned}
\int_{\mathbb{R}} \frac{\min\{SC\,z_1^2,L|z_1|\}}{|z_1|^{2s+1}}dz_1
\leq \,& SC \int_{|z_1|\leq 1} |z_1|^{1-2s}dz_1 + L \int_{|z_1|>1} |z_1|^{-2s}dz_1\\
  = \,& \frac{SC}{1-s} + \frac{2L}{2s-1}.
\end{aligned}
\end{equation}
Combining \eqref{eq:handwritten-integral-split} and \eqref{GJ-21}, we obtain
\begin{equation}
\label{GJ-22}
J-\frac{\mathcal J_x(M)}{(n+2s)\mu_1^2} \leq \frac{C}{(1-s)\mu_1^{2+2s}},
\end{equation}
where the constant $C$ depends only on $n$, $s$, $SC$ and $L$ and remains bounded as $s\to1$.
Noticing that $\mu_1 = m^{-1/2}$, we derive \eqref{eq:directional-bound} from \eqref{eq:directional-cancellation}
and \eqref{GJ-22}.

Finally, the homogeneity of the kernel in $\mathcal J_x(M)$ gives
\[
 \mathcal J_x(tM)=t^s\mathcal J_x(M),\qquad t>0.
\]
Differentiating at $t=1$ gives
$\operatorname{trace}\big(D\mathcal J_x(M)M\big)=s\mathcal J_x(M)$,
proving \eqref{eq:J-homogeneity}.
\end{proof}

From now on, we consider the case $f = \sigma_k^{1/k}$. For simplicity, we still denote $f = \sigma_k^{1/k}$ and $F(T)=\sigma_k^{1/k}(\lambda(T))$. Set
\[
 \Sigma_k=\{T\in\Gamma_k:\sigma_k(T)=1\}.
\]
By Lemma~\ref{lem:normalized-matrix-family},
\[
 \mathcal M_k=\{DF(T):T\in\Sigma_k\}.
\]

\begin{lemma}\label{lem:matrix-preimage-compactness}
For each compact subset $\mathcal{K}$ of $\mathcal{S}^{n\times n}$, the set
\[
\mathcal{K}^{-1}:=\{T\in\Sigma_k:DF(T)\in \mathcal{K}\}
\]
is compact in $\Gamma_k$.
\end{lemma}

\begin{proof}
Since $F$ is smooth, the set $\mathcal{K}^{-1}$ is closed in $\Gamma_k$.
It suffices to prove that $\mathcal{K}^{-1}$ is bounded. Since $\mathcal{K}$ is compact,
there exists a positive constant $C_K$ such that
\[
 \sum_{i=1}^n M^{ii}\leq C_K,\mbox{ for all } M=\{M^{ij}\}\in \mathcal{K}.
\]
Let $T$ be an arbitrary matrix in $\mathcal{K}^{-1}$. We have
$\sum_{i=1}^nF^{ii}(T)\leq C_K$.
Let $\xi=(\xi_1,\ldots,\xi_n)$ be the eigenvalues of $T$.
Since $\sigma_k(\xi)=1$, we find
\begin{equation}\label{eq:trace-DF}
 \sum_{i=1}^n F^{ii}(T)=\frac{n-k+1}{k}\sigma_{k-1}(\xi).
\end{equation}
By the Newton-Maclaurin inequality, we have
\[
 \sigma_{k-1}(\xi)
 \geq C_0\sigma_1^{1/(k-1)}(\xi)
 \sigma_k^{1-1/(k-1)}(\xi)
 =C_0\sigma_1^{1/(k-1)}(\xi)
\]
for some positive constant $C_0$ depending only on $n$ and $k$.
It follows that
\[
 \begin{aligned}
 \operatorname{trace}(T)=\sigma_1(\xi)
 &\leq\left[\frac{k}{n-k+1}C_0^{-1}
 \sum_{i=1}^nF^{ii}(T)\right]^{k-1}\\
 &\leq\left(\frac{k}{n-k+1}C_0^{-1}C_K\right)^{k-1}.
 \end{aligned}
\]
Since $T\in\Gamma_k\subset\Gamma_2$, we see that
\[
 \sigma_2(\xi)=\frac12\big(\sigma_1^2(\xi)-|T|^2\big)>0
\]
and
\[
 |T|^2<\sigma_1^2(\xi)
 \leq\left(\frac{k}{n-k+1}C_0^{-1}C_K\right)^{2(k-1)}.
\]
Thus, $\mathcal{K}^{-1}$ is bounded and then is compact.
\end{proof}
\begin{remark}
\label{GJ-rem1}
Lemma \ref{lem:matrix-preimage-compactness} also holds for the operator $\tilde{f}_k$ defined in \eqref{GJ-24}.
In fact, in this case,
\[
\sum F^{ii} (T) = \sum_{i=1}^n F^{ii}(T)=\frac{(n-1)(n-k+1)}{k}\sigma_{k-1}(\eta),
\]
where $\eta$ is the vector given by
\begin{equation}
\label{eta}
\eta = \Big(\sum_{j\neq 1} \xi_j, \ldots, \sum_{j\neq n} \xi_j\Big).
\end{equation}
By the Newton-Maclaurin inequality again,
\[
 \operatorname{trace}(T)= \frac{1}{n-1}\sigma_1(\eta)
 \leq\frac{1}{n-1} \left[\frac{k}{(n-1)(n-k+1)}C_0^{-1}
 \sum_{i=1}^nF^{ii}(T)\right]^{k-1}.
\]
A bound for $|T|^2$ can be proved similarly by noticing that $\sigma_1(\xi) I - T \in \Gamma_k \subset \Gamma_2$.
\end{remark}

The following lemma was proved by Chou-Wang \cite{ChouWang2001} and Hou-Ma-Wu \cite{HouMaWu2010}.
\begin{lemma}\label{lem:eigenvalue-inequality}
Let $\xi\in\Gamma_k$ and $\xi_1\geq\cdots\geq\xi_n$.
Then we have
\begin{equation}\label{eq:largest-eigenvalue-inequality}
 \xi_1\sigma_{k-1;1}(\xi)\geq\frac{k}{n}\sigma_k(\xi).
\end{equation}
\end{lemma}
A direct corollary of Lemma~\ref{lem:eigenvalue-inequality} is that
\begin{equation}\label{eq:tau-lower}
 \lambda_{\max}(T)\geq
 \frac{1}{n\,\lambda_{\min}(DF(T))},\mbox{ for all } T\in\Sigma_k.
\end{equation}
Another basic property of Hessian operator is that every eigenvector $e$ satisfying
$Me=\lambda_{\min}(M)e$ also satisfies
\begin{equation}
\label{GJ-26}
Te=\lambda_{\max}(T)e.
\end{equation}
We also need the following lemma to describe the kernel of $D^2 F$.
\begin{remark}
\label{GJ-re2}
\eqref{eq:tau-lower} also holds for $f = \tilde{f}_k$. Indeed, for $\xi\in\tilde{\Gamma}_k\cap \{\xi: \tilde{f}_k (\xi)=1\}$ and $\xi_1\geq\cdots\geq\xi_n$,
we have
\[
k=\sum_{j=1}^n \eta_j \sigma_{k-1;j}(\eta)=\eta_1\sigma_{k-1;1}(\eta)+\sum_{j\neq 1}\eta_j \sigma_{k-1;j}(\eta)\leq n \xi_1 \sum_{j\neq 1}\sigma_{k-1;j}(\eta).
\]
So, we see
\[
\xi_1\frac{\partial\tilde{f}_k}{\partial\xi_1} = \frac{1}{k} \xi_1 \sum_{j\neq 1} \sigma_{k-1;j} (\eta) \geq \frac{1}{n}
\]
where $\eta$ is defined in \eqref{eta}.
Thus, \eqref{eq:tau-lower} holds for $f = \tilde{f}_k$.
\end{remark}
\begin{lemma}\label{lem:hessian-kernel}
Let $f=\sigma_k^{1/k}$, $2\leq k\leq n$. Then,
\begin{equation}
\label{GJ-23}
 \ker D^2 F(T)=\{aT: a\in\mathbb R\},\mbox{ for all } T\in\Gamma_k.
\end{equation}
\end{lemma}
Lemma \ref{lem:hessian-kernel} is an easy corollary of Theorem 3.2 in \cite{Gerhardt2015} where a Hessian
operator is called strictly concave if the multiplicity of the eigenvalue $0$ for $D^2 f(\lambda)$ is one for all $\lambda \in \Gamma$.
In Theorem 3.2 of \cite{Gerhardt2015}, Gerhardt proved that $f=\sigma_k^{1/k}$ is strictly concave. It follows that the dimension $\ker D^2 f(\lambda)$
is $1$ for any $\lambda \in \Gamma_k$. By \eqref{f4}, we see $f (t\lambda) = t f(\lambda)$ for any $t > 0$. Differentiating it twice, we get $f_{ij} (\lambda) \lambda_i \lambda_j = 0$
for all $\lambda \in \Gamma$. It means $\lambda \in \ker D^2 f(\lambda)$ and therefore,
\[
\ker D^2 f (\lambda) = \{a \lambda: a \in \mathbb{R}\}, \mbox{ for } f = \sigma_k^{1/k}, \lambda \in \Gamma_k.
\]
Thus, \eqref{GJ-23} follows. More generally, \eqref{GJ-23} holds for all strictly concave Hessian operators in the sense of \cite{Gerhardt2015}.
It is easy to find that $f$ is strictly concave if and only if for any $\sigma>0$,
the level hypersurface
\[
 \Gamma^\sigma=\{\xi\in\Gamma_k:f(\xi)=\sigma\}
\]
is locally strictly convex. In fact, by direct calculations, we find the second fundamental form of $\Gamma^\sigma$ with respect to the normal $\nu = \frac{Df}{|Df|}$ is
$-\frac{D^2f}{|Df|}$. Furthermore, $\lambda \notin T_\lambda \Gamma^\sigma$ since $\lambda \cdot \nu = \frac{f(\lambda)}{|Df|} \neq 0$ by \eqref{f4}.

In this paper, we provide a slightly more generalized proof for Theorem 3.2 in \cite{Gerhardt2015} using the theory for hyperbolic polynomials by G{\aa}rding \cite{Garding59}.
We first recall the definition of hyperbolic polynomials.
\begin{definition}
Let $P=P(\lambda)$ be a homogeneous polynomial of degree $k>0$ in $n$ variables and let $a\in \mathbb{R}^n$. We say that $P$ is hyperbolic with respect to $a$ or, that $P$ is hyperbolic $a$,
if the equation $P(sa+v)=0$ has $k$ real zeros for every $v\in\mathbb{R}^n$.

The cone $C = C(P,a)$ is defined to be the set of all $\lambda$ such that $P(ta+\lambda) \neq 0$ when $t\geq 0$.

$LP$ is the linear subspace of $\mathbb{R}^n$ defined as the set of all $\lambda$
such that $P(t\lambda+\mu)=P(\mu)$ for all $t$ and $\mu$.
\end{definition}
\begin{proposition}
\label{GJ-prop2}
Let $P$ be hyperbolic $a$, $P(a) > 0$ and $k>1$. Then
\[
\ker D^2 P^{1/k}(\lambda) = \mathrm{span}\{\lambda\} \oplus LP, \ \mbox{ for } \lambda \in C(P,a).
\]
\end{proposition}
\begin{proof}
Let $\lambda \in C (P,a)$. By Theorem 2 of \cite{Garding59}, $P$ is also hyperbolic $\lambda$. Thus, for any $v\in\mathbb{R}^n$,
$P(t\lambda+v)=0$ has $k$ real zeros. It follows that $P(\lambda + tv)$ also has $k$ real zeros since $P$ is homogeneous. Let
\[
P(\lambda + tv) = P (\lambda) \prod_{j=1}^k (1 + t s_j), \mbox{ where } s_j \in \mathbb{R}, j = 1, \ldots, k.
\]
Let $f = P^{1/k}$. It follows that
\[
f(\lambda + tv) = f (\lambda) \prod_{j=1}^k (1 + t s_j)^{1/k}.
\]
Therefore,
\[
\frac{d^2}{dt^2} f(\lambda + tv)\Big|_{t=0} = - \frac{f(\lambda)}{k^2} \sum_{i<j} (s_i-s_j)^2.
\]
On the other hand,
\[
\frac{d^2}{dt^2} f(\lambda + tv)\Big|_{t=0} = v^T D^2 f(\lambda) v.
\]
Thus,
\begin{equation}
\label{GJ-25}
v\in \ker D^2 f(\lambda) \mbox{ if and only if } s_1 =\cdots s_k = s,
\end{equation}
which means
\[
P(\lambda + tv) = P (\lambda) (1 + t s)^k.
\]
We then obtain
\[
P \left(\lambda + \frac{t}{1+ts} (v-s\lambda)\right) = P(\lambda), \mbox{ for all }t.
\]
Consequently,
\[
P (\lambda + t (v-s\lambda)) = P(\lambda), \mbox{ for all }t,
\]
which means that the roots of $P (t\lambda + (v-s\lambda)) = 0$ are all $0$. By Theorem 3 of \cite{Garding59}, we find $v-s\lambda \in LP$
and thus, $v \in \mathrm{span}\{\lambda\} \oplus LP$. We have proved that $\ker D^2 P^{1/k}(\lambda) \subset \mathrm{span}\{\lambda\} \oplus LP$.

Conversely, fix any $v = s\lambda + w \in \mathrm{span}\{\lambda\} \oplus LP$, where $s\in \mathbb{R}$ and $w\in LP$. By the definition of $LP$, we have
\[
P(\lambda + tv) = P ((1+ts)\lambda + tw) = P((1+ts)\lambda) = (1+ts)^k P(\lambda).
\]
By \eqref{GJ-25}, we have $v\in \ker D^2 f(\lambda)$, which means $\mathrm{span}\{\lambda\} \oplus LP \subset \ker D^2 f(\lambda)$.
Proposition \ref{GJ-prop2} is then proved.
\end{proof}
It was shown in \cite{Garding59} that $\sigma_k$ is hyperbolic $\textbf{1}$ for
$\textbf{1} = (1,\ldots, 1)\in \mathbb{R}^n$ and $C(\sigma_k, \textbf{1}) = \Gamma_k$.
Obviously, $L\sigma_k = \{0\}$ for $2\leq k \leq n$. Therefore, $\ker D^2 \sigma_k^{1/k}(\lambda) = \mathrm{span}\{\lambda\}$ for all $\lambda\in \Gamma_k$
with $2\leq k \leq n$.
\begin{corollary}
\label{GJ-cor1}
Let $\tilde{f}_k$ be the operator defined by \eqref{GJ-24}. Then
\[
\ker D^2 \tilde{f}_k (\lambda) = \mathrm{span}\{\lambda\}, \ \mbox{ for } \lambda \in \tilde{\Gamma}_k.
\]
\end{corollary}
\begin{proof}
Define
\[
P (\lambda) = \sigma_k (\sigma_1(\lambda) - \lambda_1, \cdots, \sigma_1(\lambda) - \lambda_n),
\]
namely, $\tilde{f}_k (\lambda) = P^{1/k}(\lambda)$, $2\leq k \leq n$. Since $\sigma_k$ is hyperbolic $\textbf{1}$,
$\sigma_k(s\textbf{1}+v)=0$ has $k$ real zeros for every $v\in\mathbb{R}^n$. Thus,
\[
P (s\textbf{1}+v) = \sigma_k[(n-1)s\textbf{1}+(\sigma_1(v)\textbf{1}-v)]=0
\]
has $k$ real zeros. Therefore, $P$ is hyperbolic $\textbf{1}$.

Next, we find $P(t\textbf{1}+\lambda) \neq 0$ if and only if
$\sigma_k[(n-1)t\textbf{1}+(\sigma_1(\lambda)\textbf{1}-\lambda)]\neq 0$. Then $\lambda \in C (P,\textbf{1})$ if and only if
$\sigma_1(\lambda)\textbf{1}-\lambda \in C (\sigma_k,\textbf{1}) = \Gamma_k$. It follows that $C (P,\textbf{1}) = \tilde{\Gamma}_k$.

Since $P (t\lambda + \mu) = \sigma_k[t(\sigma_1(\lambda)\textbf{1}-\lambda) + (\sigma_1(\mu)\textbf{1}-\mu)]$ for $\lambda, \mu\in\mathbb{R}^n$ and
$t\in\mathbb{R}$, we have $LP=\{0\}$ by noticing that $L\sigma_k = \{0\}$.

Then Corollary \ref{GJ-cor1} follows immediately from Proposition \ref{GJ-prop2}.
\end{proof}
\begin{proof}[Proof of Theorem \ref{GJ-thm2}]
Fix $x\in\Omega$. We introduce the penalty function
\begin{equation}\label{eq:penalty}
 \Psi(M)=\log(1+\operatorname{trace}M)
 +\log(1+\operatorname{trace}M^{-1}), \mbox{ for } M\in \mathcal{S}^{n\times n}_+.
\end{equation}
For $\rho>0$, define the functional
\begin{equation}\label{eq:penalized-functional}
 G_\rho(T):=\mathcal J_x(DF(T))+\rho\Psi(DF(T)),
 \qquad T\in\Sigma_k.
\end{equation}
We claim that $G_\rho$ attains its minimum on $\Sigma_k$. For
$A\in\mathbb R$, set
\[
 \mathcal{E}_A:=\{T\in\Sigma_k:G_\rho(T)\leq A\}.
\]
Taking $A=G_\rho(F(I)^{-1}I)$ makes $\mathcal{E}_A$ nonempty.
By \eqref{GJ-7},
\[
 \mathcal{E}_A\subset \bigl(DF|_{\Sigma_k}\bigr)^{-1}\left(
 \left\{M>0:\Psi(M)\leq
 \frac{1}{\rho}\left(A-\frac{\eta_0}{1-s}\right)\right\}\right).
\]
Let $C=\rho^{-1}(A-\eta_0/(1-s))>0$.
By \eqref{eq:penalty}, every matrix in the set in braces satisfies
\[
 (e^C-1)^{-1}I\leq M\leq(e^C-1)I.
\]
Together with the continuity of $\Psi$, these bounds show that this set
is compact in $\mathcal{S}^{n\times n}$. Hence, by
Lemma~\ref{lem:matrix-preimage-compactness} we see that the set on the right hand side is
compact in $\Sigma_k$. Since $G_\rho$ is continuous, $\mathcal{E}_A$ is closed
in $\Sigma_k$ and therefore compact.
Thus there exists $T_\rho\in \mathcal{E}_A$ such that
\[
 \begin{aligned}
 G_\rho(T_\rho)&=\min_{T\in \mathcal{E}_A}G_\rho(T)\leq A,\\
 G_\rho(T)&>A\geq G_\rho(T_\rho),\ T\in\Sigma_k\setminus \mathcal{E}_A.
 \end{aligned}
\]
Hence
\[
 G_\rho(T_\rho)=\min_{T\in\Sigma_k}G_\rho(T).
\]
Let $M_\rho=DF(T_\rho)$.
We extend $G_\rho$ to $\Gamma_k$ by \eqref{eq:penalized-functional}.
Since $DF(tT)=DF(T)$ for $t>0$ and $T/F(T)\in\Sigma_k$ for
$T\in\Gamma_k$, we have
\[
 G_\rho(T)=G_\rho\!\left(\frac{T}{F(T)}\right)
 \geq G_\rho(T_\rho),\qquad T\in\Gamma_k.
\]
Thus $T_\rho$ is also a minimum point of $G_\rho$ in the cone
$\Gamma_k$, and $DG_\rho(T_\rho)=0$.
By the chain rule, for each symmetric matrix $H$,
\[
 \operatorname{trace}\left(
 \big(D\mathcal J_x(M_\rho)+\rho D\Psi(M_\rho)\big)
 D^2F(T_\rho)[H]\right)=0.
\]
By Lemma~\ref{lem:hessian-kernel}, there exists a real number
$\alpha_\rho$ such that
\begin{equation}\label{eq:EL}
 D\mathcal J_x(M_\rho)+\rho D\Psi(M_\rho)
 =\alpha_\rho T_\rho.
\end{equation}
Since
\[
 \operatorname{trace}(T_\rho M_\rho)=F(T_\rho)=1,
 \qquad
 \operatorname{trace}\big(D\mathcal J_x(M_\rho)M_\rho\big)
 =s\mathcal J_x(M_\rho),
\]
taking the trace after multiplying \eqref{eq:EL} by $M_\rho$ we get
\begin{equation}\label{eq:alpha-formula}
 \alpha_\rho=s\mathcal J_x(M_\rho)
 +\rho\operatorname{trace}\big(D\Psi(M_\rho)M_\rho\big).
\end{equation}
By direct calculation we get
\begin{equation}\label{eq:DPsi-formula}
 D\Psi(M)=\frac{I}{1+\operatorname{trace}M}
 -\frac{M^{-2}}{1+\operatorname{trace}M^{-1}}.
\end{equation}
Therefore,
\[
 \operatorname{trace}\big(D\Psi(M)M\big)
 =\frac{\operatorname{trace}M}{1+\operatorname{trace}M}
 -\frac{\operatorname{trace}M^{-1}}{1+\operatorname{trace}M^{-1}},
\]
and we then get
\[
 \left|\operatorname{trace}\big(D\Psi(M)M\big)\right|<1.
\]
It follows from
\eqref{GJ-7} and \eqref{eq:alpha-formula} that
\begin{equation}\label{eq:alpha-lower}
 \alpha_\rho\geq\frac{s\eta_0}{1-s}-\rho
 \geq\frac{s\eta_0}{2(1-s)}
\end{equation}
provided $\rho\leq s\eta_0/(2(1-s))$.

We next derive a uniform lower bound for $M_\rho$.
Choose $\rho_0>0$ such that
\[
 \rho_0\leq\min\left\{1,
 \frac{s\eta_0}{2(1-s)},
 \frac{s\eta_0}{4n(1-s)}\right\}.
\]
For $0<\rho\leq\rho_0$, let
$m_\rho:=\lambda_{\min}(M_\rho)$.
Suppose $m_\rho\leq1$. We choose a unit eigenvector $e$ such that
$M_\rho e=m_\rho e$. By \eqref{eq:DPsi-formula} we find
\begin{equation}
\label{GJ-27}
e^TD\Psi(M_\rho)e\leq1.
\end{equation}
Therefore, by \eqref{eq:tau-lower}, \eqref{GJ-26}, \eqref{eq:EL} and \eqref{eq:directional-bound}, we get
\begin{equation}
 \label{eq:direct-lower-bound}
\begin{aligned}
 \frac{s\eta_0}{2n(1-s)m_\rho}
 &\leq\alpha_\rho\lambda_{\max}(T_\rho) = \alpha_\rho e^T T_\rho e\\
 &=e^TD\mathcal J_x(M_\rho)e+\rho e^TD\Psi(M_\rho)e\\
 &\leq \frac{C m_\rho^{s-1}}{1-s}+\rho.
\end{aligned}
\end{equation}
Multiplying by $m_\rho$ and using $m_\rho\leq1$ and the choice of
$\rho_0$, we obtain
\[
 \frac{C}{1-s} m_\rho^s\geq\frac{s\eta_0}{2n(1-s)}-\rho m_\rho
 \geq\frac{s\eta_0}{4n(1-s)}.
\]
Together with the case $m_\rho>1$, this gives a uniform
positive lower bound for $m_\rho$.
Therefore,
\begin{equation}\label{eq:lower-Mrho}
 M_\rho\geq\theta I,
 \qquad 0<\rho\leq\rho_0.
\end{equation}
where
\begin{equation}\label{eq:theta-choice}
 \theta=\frac12\min\left\{
 1,
 \left(\frac{s\eta_0}{4nC}\right)^{1/s}
 \right\}>0,
\end{equation}

Choose a unit vector $e$ such that
$M_\rho e=m_\rho e$. By \eqref{GJ-26},
$T_\rho e=\tau_\rho e$, where
$\tau_\rho=\lambda_{\max}(T_\rho)$. It follows from
\eqref{eq:directional-bound}, \eqref{eq:EL}, \eqref{GJ-27}, and
\eqref{eq:alpha-lower} that
\[
 \frac{s\eta_0}{2(1-s)}\tau_\rho
 \leq \frac{C\theta^{s-1}}{1-s}+\rho_0.
\]
Thus, $\tau_\rho$ is uniformly bounded. Since
$T_\rho\in\Gamma_k\subset\Gamma_2$,
\[
 |T_\rho|^2
 =\sigma_1(T_\rho)^2-2\sigma_2(T_\rho)
 <\sigma_1(T_\rho)^2
 \leq n^2\tau_\rho^2.
\]
Note that
\[
\operatorname{trace}(DF(T_\rho)) = \frac{n-k+1}{k}\sigma_{k-1}(\lambda(T_\rho)).
\]
We have
\[
 |DF(T_\rho)|\leq C(n,k)|T_\rho|^{k-1}.
\]
Thus there exists $\Lambda>0$, depending only on
$n,k,s,\eta_0,L$ and $SC$, such that
\begin{equation}\label{eq:two-sided-Mrho}
 \theta I\leq M_\rho\leq\Lambda I,
 \qquad 0<\rho\leq\rho_0.
\end{equation}

Let $\rho_j\to0$. By \eqref{eq:two-sided-Mrho},
$M_{\rho_j}\in \mathcal{K}:=\{M:\theta I\leq M\leq\Lambda I\}$.
Since $\bigl(DF|_{\Sigma_k}\bigr)^{-1}(\mathcal{K})$ is compact, passing to a subsequence yields
$T_{\rho_j}\to T_x\in\Sigma_k$. By the continuity of $DF$,
\[
 M_{\rho_j}=DF(T_{\rho_j})\to DF(T_x)=:M_x\in\mathcal M_k,
 \qquad \theta I\leq M_x\leq\Lambda I.
\]
Given $\varepsilon>0$, choose
$\widehat M\in\mathcal M_k$ such that
\[
 \mathcal J_x(\widehat M)\leq F_s[u](x)+\varepsilon.
\]
Since $T_{\rho_j}$ minimizes $G_{\rho_j}$, we have
\[
 \mathcal J_x(M_{\rho_j})+\rho_j\Psi(M_{\rho_j})
 \leq\mathcal J_x(\widehat M)+\rho_j\Psi(\widehat M).
\]
Combining the preceding inequalities and using $\Psi\geq0$, we have
\[
 F_s[u](x)\leq\mathcal J_x(M_{\rho_j})
 \leq F_s[u](x)+\varepsilon+\rho_j\Psi(\widehat M).
\]
By Lemma \ref{lem:bellman-variation}, letting $j\to\infty$ we obtain
\[
F_s[u](x)\leq \mathcal J_x(M_x) \leq F_s[u](x) + \varepsilon.
\]
Since $\varepsilon$ can be chosen arbitrarily, we have
\[
 \mathcal J_x(M_x)=F_s[u](x),
 \qquad \theta I\leq M_x\leq\Lambda I.
\]
It follows that
\[
 F_s[u](x) \leq F_s^{\theta,\Lambda}[u](x)
 \leq\mathcal J_x(M_x)=F_s[u](x).
\]
Thus, $F_s[u](x)=F_s^{\theta,\Lambda}[u](x)$.
Since $x\in\Omega$ is arbitrary, \eqref{GJ-15} follows.
All constants depend only on $n$, $k$, $s$, $\eta_0$,
$L$, and $SC$.

Finally, noticing that$\theta$ remains uniformly bounded away
from zero and $\Lambda$ remains uniformly bounded above as
$s\to1^-$, Theorem \ref{GJ-thm2} is proved.
\end{proof}

\begin{remark}
\label{GJ-rem2}
By Remark \ref{GJ-rem1} and Corollary \eqref{GJ-cor1}, all the lemmas in this section hold for $\tilde{f}_k$. Then Theorem \ref{GJ-thm2} holds for $\tilde{f}_k$.
\end{remark}


\begin{thebibliography}{99}



\bibitem{caffarelli2015} L. A. Caffarelli and F. Charro, {\em On a fractional Monge-Amp\`{e}re operator},
	Annals of PDE, {\bf 1} (2015), no. 1, 1-47.


\bibitem{caffarelli1985} L. A. Caffarelli, L. Nirenberg and J. Spruck, {\em The Dirichlet problem for nonlinear second-order elliptic equations III: Functions of eigenvalues of the Hessians}, Acta Mathematica {\bf 155} (1985), 261-301.

\bibitem{caffarelli2009} L. A. Caffarelli and L. Silvestre, {\em Regularity theory for fully nonlinear integro-differential equations},
	Communications on Pure and Applied Mathematics, {\bf 62} (2009), 597-638.

\bibitem{caffarelli2011} L. A. Caffarelli and L. Silvestre, {\em The Evans-Krylov theorem for nonlocal fully nonlinear equations},
    Annals of Mathematics, {\bf 174} (2011), 1163-1187.

\bibitem{caffarelli2016} L. A. Caffarelli and L. Silvestre, {\em A nonlocal Monge-Amp\`{e}re equation},
	Communications in Analysis and Geometry, {\bf 24} (2016), 307-335.

\bibitem{caffarelli2024} L. A. Caffarelli and M. Soria-Carro, {\em On a family of fully nonlinear integrodifferential operators: from fractional laplacian to nonlocal Monge-Amp\`{e}re}, Analysis and PDE, {\bf 17} (2024), 243-279.

\bibitem{chen2020} W. Chen, Y. Li and P. Ma, {\em The Fractional Laplacian}, World Scientific Publishing, Singapore, 2020.

\bibitem{ChouWang2001} K.-S. Chou and X.-J. Wang,  {\em A variational theory of the Hessian equation},
    Communications on Pure and Applied Mathematics, {\bf 54} (2001), 1029--1064.

\bibitem{Garding59} L. G{\aa}rding, {\em An inequality for hyperbolic polynomials}, J. Math. Mech., {\bf 8} (1959), 957-965.


\bibitem{Gerhardt2015} C. Gerhardt, {\em Curvature flows in the sphere},
    Journal of Differential Geometry, {\bf 100} (2015), no. 2, 301--347.

\bibitem{Guan94}B. Guan {\em The Dirichlet problem for a class of fully nonlinear elliptic equations}, Communications in Partial Differential Equations, {\bf 19} (1994),399-416.

\bibitem{Guan23}B. Guan {\em The Dirichlet problem for fully nonlinear elliptic equations on Riemannian manifolds}, Advances in Mathematics, {\bf 415} (2023), 108899, 32 pp.


\bibitem{GuillenSchwab2012} N. Guillen and R. W. Schwab, {\em Aleksandrov--Bakelman--Pucci type estimates for integro-differential equations},
	Archive for Rational Mechanics and Analysis, {\bf 206} (2012), 111-157.

\bibitem{HouMaWu2010} Z. Hou, X.-N. Ma and D. Wu, {\em A second order estimate for complex Hessian equations on a compact K\"ahler manifold},
    Mathematical Research Letters, {\bf 17} (2010), no. 3, 547--561.

\bibitem{Ivochkina81} N. M. Ivochkina, {\em The integral method of barrier functions and the Dirichlet problem for equations with operators
    of the Monge-Amp\`{e}re type}, (Russian) Matematicheski{\u i} Sbornik, (N.S.) {\bf 112} (1980), 193-206. [English transl.: Mathematics of the USSR Sbornik, {\bf 40} (1981), 179-192.]

\bibitem{LT94} M. Lin and N. S. Trudinger, {\em On some inequalities for elementary symmetric functions},
	Bull. Austral. Math. Soc. {\bf 50} (1994), 317-326.



%

\bibitem{Spruck05} J. Spruck, {\em Geometric aspects of the theory of fully nonlinear elliptic equations},
	Clay Mathematics Proceedings, {\bf 2}, (2005), 283-309.

\bibitem{Trudinger90} N. S. Trudinger, {\em The Dirichlet Problem for the Prescribed Curvature Equations},
	Archive for Rational Mechanics and Analysis, {\bf 111} (1990), 153-179.

\bibitem{Trudinger95} N. S. Trudinger, {\em On the Dirichlet problem for Hessian equations},
	Acta Mathematica, {\bf 175} (1995), 151-164.

%

\bibitem{Wang1994} X. J. Wang, {\em A class of fully nonlinear elliptic equations and related
	functionals}, Indiana University Mathematics Journal, {\bf 43} (1994), 25-54.

\bibitem{wu2019} Y. Wu, {\em Regularity of fractional analogue of $k$-Hessian operators and a non-local one-phase free boundary problem},
	available at https://www.proquest.com/docview/247343852. PhD thesis. University of Texas at Austin, 2019.





\end{thebibliography}
\end{document}